\documentclass[english,a4paper]{article}
\usepackage[T1]{fontenc}
\usepackage[latin9]{inputenc}
\usepackage{amsthm}
\usepackage{amsmath}
\usepackage{amssymb}
\usepackage{esint}

\makeatletter
\theoremstyle{plain}
\newtheorem{thm}{Theorem}
  \theoremstyle{remark}
  \newtheorem{rem}[thm]{Remark}
  \theoremstyle{plain}
  \newtheorem{lem}[thm]{Lemma}
  \theoremstyle{plain}
  \newtheorem{cor}[thm]{Corollary}
  \theoremstyle{plain}
  \newtheorem{prop}[thm]{Proposition}
\usepackage{mathrsfs}

\DeclareMathOperator{\Id}{Id}

\usepackage{babel}

\begin{document}

\title{Non degeneracy for solutions of singularly perturbed nonlinear elliptic problems
on symmetric Riemannian manifolds}

\author{M.Ghimenti\thanks{Dipartimento di Matematica Applicata,
Universit\`a di Pisa, via Buonarroti 1c, 56127, Pisa, Italy, e-mail
{\tt marco.ghimenti@dma.unipi.it, a.micheletti@dma.unipi.it} },
A.M.Micheletti\addtocounter{footnote}{-1}\footnotemark }

\maketitle
\begin{abstract}
Given a symmetric Riemannian manifold $(M,g)$, we show some results
of genericity for non degenerate sign changing solutions of singularly
perturbed nonlinear elliptic problems with respect to the parameters:
the positive number $\varepsilon$ and the symmetric metric $g$.
Using these results we obtain a lower bound on the number of non degenerate
solutions which change sign exactly once.


\end{abstract}

\section{Introduction}

Let $(M,g)$ be a smooth connected compact Riemannian manifold of
finite dimension $n\geq2$ embedded in $\mathbb{R}^{N}$. Le us consider
the problem\begin{equation}
\left\{ \begin{array}{cc}
-\varepsilon^{2}\Delta_{g}u+u=|u|^{p-2}u & \text{ in }M\\
u\in H_{g}^{1}(M)\end{array}\right.\label{eq:P}\end{equation}

Recently there have been some results on the influence of the topology
(see \cite{BBM07,H09,V08}) and the geometry (see \cite{BP05,DMP09,MP09})
of $M$ on the number of positive solutions of problem (\ref{eq:P}).
This problem has similar features with the Neumann problem on a flat
domain, which has been largely studied in literature (see \cite{DY,DFW,G,GWW,Li,NT1,NT2,W1,WW05,WW}). 

Concerning the sign changing solution the first result is contained
in \cite{MP09b} where it is showed the existence of a solution with
one positive peak and one negative peak when the scalar curvature
of $(M,g)$ is non constant.

Moreover in \cite{GM10} the authors give a multiplicity result for
solutions which change sign exactly once when the Riemannian manifold
is symmetric with respect to an orthogonal involution $\tau$ using
the equivariant Ljusternik Schnirelmann category.

In this paper we are interested in studying the non degeneracy of
changing sign solutions when the Riemannian manifold $(M,g)$ is symmetric.

We consider the problem
\begin{equation}
\left\{ \begin{array}{cc}
-\varepsilon^{2}\Delta_{g}u+u=|u|^{p-2}u & u\in H_{g}^{1}(M)\\
u(\tau x)=-u(x) & \forall x\in M\end{array}\right.\label{eq:Ptau}
\end{equation}
where $\tau:\mathbb{R}^{N}\rightarrow\mathbb{R}^{N}$ is an orthogonal
linear transformation such that $\tau\neq \Id$, $\tau^{2}=\Id$ ($\Id$
being the identity on $\mathbb{R}^{N}$). Here the compact connected
Riemannian manifold $(M,g)$ of dimension $n\geq2$ is a regular submanifold
of $\mathbb{R}^{N}$ invariant with respect to $\tau$. Let $M_{\tau}=\left\{ x\in M\ :\ \tau x=x\right\} $.
In the case $M_{\tau}\neq\emptyset$ we assume that $M_{\tau}$ is
a regular submanifold of $M$. In the following 
$H_{g}^{\tau}=\left\{ u\in H_{g}^{1}(M)\ :\ \tau^*u=u\right\} $
where the linear operator $\tau^{*}:H_{g}^{1}\rightarrow H_{g}^{1}$
is $\tau^{*}u=-u(\tau(x))$. 

We obtain the following genericity results about the non degeneracy
of changing sign solutions of (\ref{eq:Ptau}) with respect to the
parameters: the positive number $\varepsilon$, and the symmetric
metric $g$ (i.e. $g(\tau x)=g(x)$).
\begin{thm}
\label{thm:residuale}Given $g_{0}\in\mathscr{M}^{k}$, the set\[
D=\left\{ \begin{array}{c}
(\varepsilon,h)\in(0,1)\times\mathscr{B}_{\rho}\ :\text{ any }u\in H_{g_{0}}^{\tau}\text{ solution of}\\
-\varepsilon^{2}\Delta_{g_{0}+h}u+u=|u|^{p-2}u\text{ is not degenerate}\end{array}\right\} \]
is a residual subset of $(0,1)\times\mathscr{B}_{\rho}$. 
\end{thm}

\begin{rem}
\label{thm:residuale2}By the previous result we prove that, given $g_{0}\in\mathscr{M}^{k}$ and $\varepsilon_{0}>0$,
the set\[
D^{*}=\left\{ \begin{array}{c}
h\in\mathscr{B}_{\rho}\ :\text{ any }u\in H_{g_{0}}^{\tau}\text{ solution of}\\
-\varepsilon^{2}\Delta_{g_{0}+h}u+u=|u|^{p-2}u\text{ is not degenerate}\end{array}\right\} \]
is a residual subset of $\mathscr{B}_{\rho}$. 
\end{rem}
In the following we set \[
m_{\varepsilon_{0},g_{0}}^{\tau}=\inf_{u\in\mathcal{N}_{\varepsilon_{0},g_{0}}^{\tau}}J_{\varepsilon_{0},g_{0}}(u)\]
where \begin{eqnarray*}
J_{\varepsilon_{0},g_{0}} & (u)= & \frac{1}{\varepsilon_{0}^{n}}\int_{M}\left[\frac{1}{2}\left(\varepsilon_{0}^{2}|\nabla_{g}u|^{2}+u^{2}\right)-\frac{1}{p}|u|^{p}\right]d\mu_{g_{0}}\\
\mathcal{N}_{\varepsilon_{0},g_{0}}^{\tau} & = & \left\{ u\in H_{g_{0}}^{\tau}(M)\smallsetminus\left\{ 0\right\} \ :\ J'_{\varepsilon_{0},g_{0}}(u)\left[u\right]=0\right\} .\end{eqnarray*}

\begin{thm}
\label{thm:ApertoDenso}Given $g_{0}\in\mathscr{M}^{k}$ and $\varepsilon_{0}>0$.
If there exists $\mu>m_{\varepsilon_{0},g_{0}}^{\tau}$ which is not
a critical level of the functional $J_{\varepsilon_{0},g_{0}}^{\tau}$,
then the set\[
D^{\dagger}=\left\{ \begin{array}{c}
h\in\mathscr{B}_{\rho}\ :\text{ any }u\in H_{g_{0+h}}^{\tau}\text{ solution of}\\
-\varepsilon^{2}\Delta_{g_{0}+h}u+u=|u|^{p-2}u\text{ with }J_{\varepsilon_{0},g_{0}}^{\tau}(u)<\mu\text{ is not degenerate}\end{array}\right\} \]
is an open dense subset of $\mathscr{B}_{\rho}$. 
\end{thm}
Here the set $\mathscr{B}_{\rho}$ is the ball centered
at $0$ with radius $\rho$  in the space $\mathscr{S}^{k}$, where $\rho$ is small enough 
and $\mathscr{S}^{k}$ is the Banach space of all $C^{k}$, $k\ge3,$
symmetric covariants $2$-tensor $h(x)$ on $M$ such that $h(x)=h(\tau x)$
for $x\in M$. $\mathscr{M}^{k}\subset\mathscr{S}^{k}$ is the
set of all $C^{k}$ Riemannian metrics $g$ on $M$ such that $g(x)=g(\tau x)$.

These results
can be applied to obtain a lower bound for the number of non degenerate solutions
of (\ref{eq:Ptau}) which change sign exactly once when $M$ is invariant with
respect to the involution $\tau=-\Id$ and $0\notin M$. 
We get the following propositions.

\begin{prop}
\label{thm:T1}Given $g_{0}\in\mathscr{M}^{k}$, the set \[
\mathscr{A}=\left\{ \begin{array}{c}
(\varepsilon,h)\in(0,\tilde{\varepsilon})\times\mathscr{B}_{\rho}:
\text{ the equation }-\varepsilon^{2}\Delta_{g_{0}+h}u+u=|u|^{p-2}u\\
\text{has at least }P_{1}(M/G)\text{ pairs of non degenerate solutions }\\
(u,-u)\in H_{g}^{\tau}\smallsetminus\left\{ 0\right\} \text{ which change sign exactly once}\end{array}\right\} \]
 is a residual subset of $(0,1)\times\mathscr{B}_{\rho}$.
\end{prop}

\begin{prop}
\label{thm:T3}Given $g_{0}\in\mathscr{M}^{k}$ and $\varepsilon_{0}>0$,
if there exists $\mu>m_{\varepsilon_{0},g_{0}}^{\tau}$ not
a critical value of $J_{\varepsilon_{0},g_{0}}$ in $H_{g_{0}}^{\tau}$,
then the set \[
\mathscr{A}^{\dagger}=\left\{ \begin{array}{c}
h\in\mathscr{B}_{\rho}:\text{ the equation }-\varepsilon_{0}^{2}\Delta_{g_{0}+h}u+u=|u|^{p-2}u\\
\text{has at least }P_{1}(M/G)\text{ pairs of non degenerate solutions }\\
(u,-u)\in H_{g}^{\tau}\smallsetminus\left\{ 0\right\} \text{ which change sign exactly once}\end{array}\right\} \]
 is an open dense subset of $\mathscr{B}_{\rho}$.
\end{prop}
Here $P_{t}(M/G)$ is the Poincar\'e polynomial of the manifold $M/G$,
where $G=\left\{ \Id,-\Id\right\} $, and $P_{1}(M/G)$ is when $t=1$.
By definition we have $P_{t}(M/G)=\sum_k \text{dim }H_k(M/G)\cdot t^k$ where 
$H_k(M/G)$ is the $k$-th homology group with coefficients in some field.

The paper is organized as follows. In Section \ref{sec:Preliminaries}
we recall some preliminary results. In Section \ref{sec:Sketch} we
sketch the proof of the results of genericity (theorems \ref{thm:residuale} and \ref{thm:ApertoDenso})
using some technical lemmas proved in Section \ref{sec:lemmas}. In
Section \ref{sec:application} we prove propositions \ref{thm:T1} and \ref{thm:T3}.

\section{\label{sec:Preliminaries}Preliminaries}

Given a connected $n$ dimensional $C^{\infty}$ compact manifold
$M$ without boundary endowed with a Riemannian metric $g$, we define
the functional spaces $L_{g}^{p}$, $L_{\varepsilon,g}^{p}$, $H_{g}^{1}$
and $H_{\varepsilon,g}^{1}$, for $2\leq p<2^{*}$ and a given $\varepsilon\in(0,1)$.
The inner products on $L_{g}^{2}$ and $H_{g}^{1}$
are, respectively\begin{eqnarray*}
\langle u,v\rangle_{L_{g}^{2}}=\int_{M}uvd\mu_{g} &  & \langle u,v\rangle_{H_{g}^{1}}=\int_{M}\left(\nabla u\nabla v+uv\right)d\mu_{g},\end{eqnarray*}
while the inner products on $L_{\varepsilon,g}^{2}$ and $H_{\varepsilon,g}^{1}$ are, respectively\begin{eqnarray*}
\langle u,v\rangle_{L_{\varepsilon,g}^{2}}=\frac{1}{\varepsilon^{n}}\int_{M}uvd\mu_{g} &  & \langle u,v\rangle_{H_{\varepsilon,g}^{1}}=\frac{1}{\varepsilon^{n}}\int_{M}\left(\varepsilon^{2}\nabla u\nabla v+uv\right)d\mu_{g}.\end{eqnarray*}
Finally, the norms in $L_{g}^{p}$ and $L_{\varepsilon,g}^{p}$ are\begin{eqnarray*}
\|u\|_{L_{g}^{p}}^{p}=\int_{M}|u|^{p}d\mu_{g} &  & \|u\|_{L_{\varepsilon,g}^{p}}^{p}=\frac{1}{\varepsilon^{n}}\int_{M}|u|^{p}d\mu_{g}.\end{eqnarray*}
We define also the space of symmetric $L^{p}$ and $H^{1}$ functions
as\begin{eqnarray*}
L_{g}^{p,\tau}=\left\{ u\in L_{g}^{p}(M)\ :\ \tau^{*}u=u\right\}  &  & H_{g}^{\tau}=\left\{ u\in H_{g}^{\tau}(M)\ :\ \tau^{*}u=u\right\} \end{eqnarray*}

As defined in the introduction, $\mathscr{S}^{k}$ is the space of
all $C^{k}$ symmetric covariants $2$-tensor $h(x)$ on $M$ such
that $h(x)=h(\tau x)$ for $x\in M$. We define a norm $\|\cdot\|_{k}$
in $\mathscr{S}^{k}$ in the following way. We fix a finite covering
$\left\{ V_{\alpha}\right\} _{\alpha\in L}$ of $M$ where $\left(V_{\alpha},\psi_{\alpha}\right)$
is an open coordinate neighborhood. If $h\in\mathscr{S}^{k}$, denoting
$h_{ij}$ the components of $h$ with respect to local coordinates
$(x_{1},\dots,x_{n})$ on $V_{\alpha}$, we define\[
\|h\|_{k}=\sum_{\alpha\in L}\ \sum_{|\beta|\leq k}\ \sum_{i,j=1}^{n}\ \sup_{\psi_{\alpha}(V_{\alpha})}\left|\frac{\partial^{\beta}h_{ij}}{\partial x_{1}^{\beta_{1}}\cdots\partial x_{n}^{\beta_{n}}}\right|.\]
The set $\mathscr{M}^{k}$ of all $C^{k}$ Riemannian metrics $g$
on $M$ such that $g(x)=g(\tau x)$ is an open set of $\mathscr{S}^{k}$. 

Given $g_{0}\in\mathscr{M}^{k}$ a symmetric Riemannian metric on
$M$, we notice that there exists $\rho>0$ (which does not depend
on $\varepsilon$ if $0<\varepsilon<1$) such that, if $h\in\mathscr{B}_{\rho}$
the sets $H_{\varepsilon,g_{0}+h}^{1}$ and $H_{\varepsilon,g_{0}}^{1}$
are the same and the two norms $\|\cdot\|_{H_{\varepsilon,g_{0}+h}^{1}}$
and $\|\cdot\|_{H_{\varepsilon,g_{0}}^{1}}$ are equivalent. The same
for $L_{\varepsilon,g_{0}+h}^{2}$ and $L_{\varepsilon,g_{0}}^{2}$.
If $h\in\mathscr{B}_{\rho}$ and $\varepsilon\in(0,1)$ we set\begin{eqnarray*}
E_{h}^{\varepsilon}(u,v)=\langle u,v\rangle_{H_{\varepsilon,g_{0}+h}^{1}} &  & \forall u,v\in H_{\varepsilon,g_{0}+h}^{1}\\
G_{h}^{\varepsilon}(u,v)=\langle u,v\rangle_{L_{\varepsilon,g_{0}+h}^{2}} &  & \forall u,v\in L_{\varepsilon,g_{0}+h}^{2}\\
N(\varepsilon,h)(u)=N_{h}^{\varepsilon}(u)=\|u\|_{\varepsilon,L_{g_{0}+h}^{p}}^{p} &  & \forall u\in L_{\varepsilon,g_{0}+h}^{p}\end{eqnarray*}

We introduce the map $A_{h}^{\varepsilon}$ which will be used in
the following section.
\begin{rem}
If $h\in\mathscr{B}_{\rho}$ and $0<\varepsilon<1$, there exists
a unique linear operator \[
A(\varepsilon,h):=A_{h}^{\varepsilon}:L_{g_{0}}^{p',\tau}(M)\rightarrow H_{g_{0}}^{\tau}\]
such that $E_{h}^{\varepsilon}(A_{h}^{\varepsilon}(u),v)=G_{h}^{\varepsilon}(u,v)$
for all $u\in L_{\varepsilon,g_{0}}^{p',\tau}$, $v\in H_{\varepsilon,g_{0}}^{\tau}$
with $2\le p<2^{*}$. Moreover $E_{h}^{\varepsilon}(A_{h}^{\varepsilon}(u),v)=E_{h}^{\varepsilon}(u,A_{h}^{\varepsilon}(v))$
for $u,v\in H_{\varepsilon,g_{0}}^{\tau}$. 

Also, we have that $A_{h}^{\varepsilon}=i_{\varepsilon,g_{0}}^{*}$
where $i_{\varepsilon,g_{0}}^{*}$ is the adjoint of the compact embedding
$i_{\varepsilon,g_{0}}:H_{\varepsilon,g_{0}}^{\tau}(M)\rightarrow L_{\varepsilon,g_{0}}^{p'\tau}(M)$
with $2\le p<2^{*}$. We recall that, if $h\in\mathscr{B}_{\rho}$
with $\rho$ small enough and $\varepsilon>0$, then $H_{\varepsilon,g_{0}}^{1}$
and $H_{\varepsilon,g_{0}+h}^{1}$ (as well as $L_{\varepsilon,g_{0}}^{p}$
and $L_{\varepsilon,g_{0}+h}^{p}$) are the same as sets and the norms
are equivalent. This is the reason why we can define $A_{h}^{\varepsilon}$
on $L_{g_{0}}^{p',\tau}$ with values in $H_{g_{0}}^{\tau}$. We summarize
some technical results contained in lemmas 2.1, 2.2 and 2.3 of \cite{MP09c}.\end{rem}
\begin{lem}
\label{lem:derivate}Let $g_{0}\in\mathscr{M}^{k}$ and $\rho$ small
enough. We have
\begin{enumerate}
\item The map $E:(0,1)\times\mathscr{B}_{\rho}\rightarrow\mathcal{L}(H_{g_{0}}^{\tau}\times H_{g_{0}}^{\tau},\mathbb{R})$
defined by $E(\varepsilon,h):=E_{h}^{\varepsilon}$ is of class $C^{1}$
and it holds, for $u,v\in H_{g_{0}}^{\tau}(M)$ and $h\in\mathscr{S}^{k}$
\begin{eqnarray*}
E'(\varepsilon_{0},h_{0})\left[\varepsilon,h\right](u,v) & = & \frac{1}{2\varepsilon_{0}^{n}}\int_{M}tr(g^{-1}h)uvd\mu_{g}+\frac{1}{\varepsilon_{0}^{n-2}}\int_{M}\langle\nabla_{g}u,\nabla_{g}v\rangle_{b(h)}d\mu_{g}\\
 &  & -\frac{n\varepsilon}{\varepsilon_{0}^{n+1}}\int_{M}uvd\mu_{g}-\frac{(n-2)\varepsilon}{\varepsilon_{0}^{n-1}}\int_{M}\langle\nabla_{g}u,\nabla_{g}v\rangle d\mu_{g}\end{eqnarray*}
with the $2$-tensor $b(h):=\frac{1}{2}tr(g^{-1}h)g-g^{-1}hg^{-1}$
\item The map $G:(0,1)\times\mathscr{B}_{\rho}\rightarrow\mathcal{L}(L_{g_{0}}^{p',\tau},H_{g_{0}}^{\tau})$
defined by $G(\varepsilon,h):=G_{h}^{\varepsilon}$ is of class $C^{1}$
and it holds, for $u,v\in H_{g_{0}}^{\tau}(M)$ and $h\in\mathscr{S}^{k}$
\begin{eqnarray*}
G'(\varepsilon_{0},h_{0})\left[\varepsilon,h\right](u,v) & = & \frac{1}{2\varepsilon_{0}^{n}}\int_{M}tr(g^{-1}h)uvd\mu_{g}-\frac{n\varepsilon}{\varepsilon_{0}^{n+1}}\int_{M}uvd\mu_{g}\end{eqnarray*}

\item The map $A:(0,1)\times\mathscr{B}_{\rho}\rightarrow\mathcal{L}(H_{g_{0}}^{\tau}\times H_{g_{0}}^{\tau},\mathbb{R})$
is of class $C^{1}$ and for any $u,v\in H_{g_{0}}^{\tau}(M)$ and
$h\in\mathscr{S}^{k}$ we have \[
E'(\varepsilon_{0},h_{0})\left[\varepsilon,h\right](A_{h_{0}}^{\varepsilon_{0}}(u),v)+E_{h_{0}}^{\varepsilon_{0}}(A'(\varepsilon_{0},h_{0})\left[\varepsilon,h\right](u),v)=G'(\varepsilon_{0},h_{0})\left[\varepsilon,h\right](u,v)\]

\item The map $N:(0,1)\times\mathscr{B}_{\rho}\rightarrow C^{0}(H_{g_{0}}^{\tau},\mathbb{R})$
defined by $(\varepsilon,h)\mapsto N_{h}^{\varepsilon}(\cdot)$ is
of class $C^{1}$ and it holds, for $u\in H_{g_{0}}^{\tau}(M)$ and
$h\in\mathscr{S}^{k}$ \begin{eqnarray*}
N'(\varepsilon_{0},h_{0})\left[\varepsilon,h\right](u) & = & \frac{1}{2\varepsilon_{0}^{n}}\int_{M}tr(g^{-1}h)|u|^{p}d\mu_{g}-\frac{n\varepsilon}{\varepsilon_{0}^{n+1}}\int_{M}|u|^{p}d\mu_{g}\end{eqnarray*}

\end{enumerate}
In all these formulas $g=g_{0}+h_{0}$ with $h_{0}\in\mathscr{B}_{\rho}$.
\end{lem}
We recall two abstract results in transversality theory (see \cite{Qu70,ST79,Uh76})
which will be fundamental for our results. 

\renewcommand{\labelenumi}{(\roman{enumi})}
\begin{thm}
\label{thm:gen1}Let $X,Y,Z$ be three real Banach spaces and let
$U\subset X,\ V\subset Y$ be two open subsets. Let $F$ be a $C^{1}$
map from $V\times U$ in to $Z$ such that 
\begin{enumerate}
\item For any $y\in V$, $F(y,\cdot):x\rightarrow F(y,x)$ is a Fredholm
map of index $0$.
\item $0$ is a regular value of $F$, that is $F'(y_{0},x_{0}):Y\times X\rightarrow Z$
is onto at any point $(y_{0},x_{0})$ such that $F(y_{0},x_{0})=0$.
\item The map $\pi\circ i:F^{-1}(0)\rightarrow Y$ is proper, where $i$
is the canonical embedding form $F^{-1}(0)$ into $Y\times X$ and
$\pi$ is the projection from $Y\times X$ onto Y
\end{enumerate}
Then the set \[
\theta=\left\{ y\in V\ :\ 0\text{ is a regular value of }F(y,\cdot)\right\} \]
 is a dense open subset of V
\end{thm}

\begin{thm}
\label{thm:gen2}If $F$ satisfies (i) and (ii) and
\begin{enumerate}
\addtocounter{enumi}{3}
\item The map $\pi\circ i$ is $\sigma$-proper,
that is $F^{-1}(0)=\cup_{s=1}^{+\infty}C_{s}$ where $C_{s}$ is a
closed set and the restriction $\pi\circ i_{|C_{s}}$ is proper for
any $s$ 
\end{enumerate}
then the set $\theta$ is a residual subset of $V$

\renewcommand{\labelenumi}{(\arabic{enumi})}
\end{thm}

\section{\label{sec:Sketch}Sketch of the proof of theorems \ref{thm:residuale} and
 \ref{thm:ApertoDenso}.}

Given $g_{0}\in\mathscr{M}^{k}$, we introduce the map $F:(0,1)\times\mathscr{B}_{\rho}\times H_{g_{0}}^{\tau}\smallsetminus\left\{ 0\right\} \rightarrow H_{g_{0}}^{\tau}$
defined by\[
F(\varepsilon,h,u)=u-A_{h}^{\varepsilon}(|u|^{p-2}u).\]
By the regularity of the map $A$ (see 3 of Lemma \ref{lem:derivate})
we get the map $F$ is of class $C^{1}$. We are going to apply transversality
Theorem \ref{thm:gen1} to the map $F$, in order to prove Theorem
\ref{thm:residuale}. In this case we have $X=H_{g_{0}}^{\tau}$,
$Y=\mathbb{R}\times\mathscr{S}^{k}$, $Z=H_{g_{0}}^{\tau}$, $U=H_{g_{0}}^{\tau}\smallsetminus\left\{ 0\right\} $
and $V=(0,1)\times\mathscr{B}_{\rho}\subset\mathbb{R}\times\mathscr{S}^{k}$. 

Assumptions (i) and (iv) are verified in Lemma \ref{lem:4.1} and
in Lemma \ref{lem:4.2}. Using Lemma \ref{lem:4.3} we can verify
(ii). 

Indeed, we have to verify that for $(\varepsilon_{0},h_{0},u_{0})\in V\times U$
such that $F(\varepsilon_{0},h_{0},u_{0})=0$ and for any $b\in H_{g_{0}}^{\tau}$,
there exists $(\varepsilon,h,v)\subset\mathscr{S}^{k}\times H_{g_{0}}^{\tau}$
such that\[
F'_{u}(\varepsilon_{0},h_{0},u_{0})\left[v\right]+F'_{\varepsilon,h}(\varepsilon_{0},h_{0},u_{0})\left[\varepsilon,h\right]=b.\]
We recall that the operator\[
v\mapsto F'_{u}(\varepsilon_{0},h_{0},u_{0})\left[v\right]=v-(p-1)i_{\varepsilon_{0},g_{0}+h}^{*}(|u_{0}|^{p-1}u_{0}v)\]
 is selfadjoint in $H_{\varepsilon_{0},g_{0}+h_{0}}^{\tau}$ and is
a Fredholm operator of index $0$. Then\[
\mathrm{Im }\ F'_{u}(\varepsilon_{0},h_{0},u_{0})\oplus\ker F'_{u}(\varepsilon_{0},h_{0},u_{0})=H_{g_{0}}^{\tau}.\]
Let $\left\{ w_{1},\dots,w_{\nu}\right\} $ be a basis of $\ker F'_{u}(\varepsilon_{0},h_{0},u_{0})\left[v\right]$.
We consider the linear functional $f_{i}:\mathbb{R}\times\mathscr{S}^{k}\rightarrow\mathbb{R}$
defined by \[
f_{i}(\varepsilon,h)=\left(F'_{\varepsilon,h}(\varepsilon_{0},h_{0},u_{0})\left[\varepsilon,h\right],
w_{i}\right)_{H_{\varepsilon_{0},g_{0}+h_{0}}^{\tau}}\ i=1,\dots,\nu.\]
By Lemma \ref{lem:4.3} we get that the linear functionals $f_{i}$
are independent. Therefore assumption (ii) is verified.
At this point by transversality theorems we get that the set
\[
\left\{ \begin{array}{c}
(\varepsilon,h)\in(0,1)\times\mathscr{B}_{\rho}\ :\text{ any }u\in H_{g_{0}}^{\tau}\smallsetminus\{0\}\text{ solution of}\\
-\varepsilon^{2}\Delta_{g_{0}+h}u+u=|u|^{p-2}u\text{ is not degenerate}\end{array}\right\} 
\]
is a residual subset of $(0,1)\times\mathscr{B}_{\rho}$. On the other hand
we observe that $0$ is a non degenerate solution of 
$-\varepsilon^{2}\Delta_{g_{0}+h}u+u=|u|^{p-2}u$, for any $\varepsilon>{0}$  and any 
$h\in \mathscr{B}_\rho$. Then, we  complete the proof of Theorem \ref{thm:residuale}.

The proof of Remark \ref{thm:residuale2} is analog to the proof
of Theorem \ref{thm:residuale} using Corollary \ref{cor:teo10}.

We now formulate the problem for Theorem \ref{thm:ApertoDenso}. Given
$g_{0}\in\mathscr{M}^{k}$ and $\varepsilon_{0}>0$, we assume that
there exists $\mu>m_{\varepsilon_{0},g_{0}}^{\tau}$ which is not
a critical level for the functional $J_{\varepsilon_{0},g_{0}}$.
It is clear that any $\mu_{0}\in(0,m_{\varepsilon_{0},g_{0}}^{\tau})$
is not a critical value of $J_{\varepsilon_{0},g_{0}}$. We set \[
\mathscr{D}=\left\{ u\in H_{g_{0}}^{\tau}\ :\ \mu_{0}<J_{\varepsilon_{0},g_{0}}(u)<\mu\right\} .\]
Now we introduce the $C^{1}$ map $H:\mathscr{B}_{\rho}\times\mathscr{A}\rightarrow H_{g_{0}}^{1}$
defined by\begin{equation}
H(h,u)=u-A_{h}^{\varepsilon_{0}}(|u|^{p-2}u)=F(\varepsilon_{0},h,u).\label{eq:G}\end{equation}
 We are going to apply transversality theorem \ref{thm:gen2} to the
map $H$. In this case $X=H_{g_{0}}^{\tau}$, $Y=\mathscr{S}^{k}$,
$Z=H_{g_{0}}^{1}(M)$, $U=\mathscr{D}\subset H_{g_{0}}^{\tau}$ and
$V=\mathscr{B}_{\rho}\subset\mathscr{S}^{k}$. It is easy to verify
assumptions (i) and (ii) for the map $H$ using Lemma \ref{lem:4.1},
Lemma \ref{lem:4.3} and Corollary \ref{cor:teo10}. Using Lemma \ref{lem:4.4}
we can verify assumption (iii) so we are in position to apply 
Theorem \ref{thm:gen2} and to get the following statement: the set
\[
\left\{ \begin{array}{c}
h\in\mathscr{B}_{\rho}\ :\text{ any }u\in H_{g_{0}}^{\tau}\text{ solution of }-\varepsilon_{0}^{2}\Delta_{g_{0}+h}u+u=|u|^{p-2}u\\
\text{such that }\mu_{0}<J_{\varepsilon_{0},g_{0}}(u)<\mu\text{ is not degenerate}\end{array}\right\} \]
is an open dense subset of $\mathscr{B}_{\rho}$. Nevertheless $0$
is a non degenerate solution of $-\varepsilon_{0}^{2}\Delta_{g_{0}+h}u+u=|u|^{p-2}u$
for any $h$, and there is no solution $u\not\equiv 0$ with $J_{\varepsilon_{0},g_{0}}(u)<\mu_0$, so we get the claim.

\section{\label{sec:lemmas}Technical lemmas}

In this section we show some lemmas in order to complete the proof of the results of 
genericity of non degenerate critical points.
\begin{lem}
\label{lem:4.1}For any $(\varepsilon,h)\in(0,1)\times\mathscr{B}_{\rho}$
the map $u\mapsto F(\varepsilon,h,u)$ with $u\in H_{g_{0}}^{\tau}$
is a Fredholm map of index zero.\end{lem}
\begin{proof}
By the definition of the map $A$, we have\[
F'(\varepsilon_{0},h_{0},u_{0})\left[v\right]=v-(p-1)A_{h_{0}}^{\varepsilon_{0}}\left[|u_{0}|^{p-2}v\right]=v-Kv,\]
where $K(v):=(p-1)i_{\varepsilon_{0},g_{0}+h_{0}}^{*}\left[|u_{0}|^{p-2}v\right]$.
We will verify that $K:H_{\varepsilon_{0},g_{0}}^{\tau}\rightarrow H_{\varepsilon_{0},g_{0}}^{\tau}$
is compact. Thus $K:H_{g_{0}}^{\tau}\rightarrow H_{g_{0}}^{\tau}$
is compact and the claim follows. In fact, in $v_{n}$ is bounded
in $H_{g_{0}}^{\tau}$, $v_{n}$ is also bounded in $H_{\varepsilon_{0},g_{0}+h_{0}}^{\tau}$
because $h_{0}\in\mathscr{B}_{\rho}$. Then, up to subsequence, $v_{n}$
converges to $v$ in $L_{\varepsilon_{0},g_{0}+h_{0}}^{t}$ for $2\leq t<2^{*}$.
So we have \[
\int_{M}\left||u_{0}|^{p-2}(v_{n}-v)\right|^{p'}d\mu_{g}\leq\left(\int_{M}|u_{0}|^{p}d\mu_{g}\right)^{\frac{p-2}{p-1}}\left(\int_{M}|v_{n}-v|^{p}d\mu_{g}\right)^{\frac{1}{p-1}}\rightarrow0.\]
 Therefore $i_{\varepsilon_{0},g_{0}+h_{0}}^{*}\left[|u_{0}|^{p-2}(v_{n}-v)\right]\rightarrow0$
in $H_{\varepsilon_{0},g_{0}+h_{0}}^{\tau}$ and also in $ $$H_{\varepsilon_{0},g_{0}}^{\tau}$. \end{proof}
\begin{lem}
\label{lem:4.2}The map $\pi\circ i:F^{-1}(0)\rightarrow\mathbb{R}\times\mathscr{S}^{k}$
is $\sigma$-proper. Here $i$ is the canonical immersion from $F^{-1}(0)$
into $\mathbb{R}\times\mathscr{S}^{k}\times H_{g_{0}}^{\tau}$ and
$\pi$ is the projection from $\mathbb{R}\times\mathscr{S}^{k}\times H_{g_{0}}^{\tau}$
into $\mathbb{R}\times\mathscr{S}^{k}$.\end{lem}
\begin{proof}
Set $I_{g_{0}}(u,R)$ the open ball in $H_{g_{0}}^{\tau}$ centered
in $u$ with radius $R$. We have ${\displaystyle F^{-1}(0)=\cup_{s=1}^{+\infty}C_{s}}$
where\[
C_{s}=\left\{ \left[\frac{1}{s},1-\frac{1}{s}\right]\times\overline{\mathscr{B}_{\rho-\frac{1}{s}}}
\times\left\{ \overline{I_{g_{0}}(0,s)}\smallsetminus I_{g_{0}}\left(0,\frac{1}{s}\right)\right\} \right\} \cap F^{-1}(0).\]
We had to prove that $\pi\circ i:C_{s}\rightarrow\mathbb{R}\times\mathscr{S}^{k}$
is proper, that is if $h_{n}\rightarrow h_{0}$ in $\overline{\mathscr{B}_{\rho-\frac{1}{s}}}$,
$\varepsilon_{n}\rightarrow\varepsilon_{0}$ in $\left[\frac{1}{s},1-\frac{1}{s}\right]$,
$u_{n}\in\left\{ \overline{I_{g_{0}}(0,s)}\smallsetminus I_{g_{0}}\left(0,\frac{1}{s}\right)\right\} $,
and $F(\varepsilon_{n},h_{n},u_{n})=0$, then, up to a subsequence,
the sequence $\left\{ u_{n}\right\} $ converges to $u_{0}\in\left\{ \overline{I_{g_{0}}(0,s)}\smallsetminus I_{g_{0}}\left(0,\frac{1}{s}\right)\right\} $.
Since $\left\{ u_{n}\right\} $ is bounded in $H_{g_{0}}^{1}$, then it
is bounded in $H_{g_{0}+h_{0}}^{1}$, since the two spaces are equivalent
because $h_{0}\in\mathscr{B}_{\rho}$. Thus $u_{n}$ converges, up
to subsequence, to $u_{0}$ in $L_{g_{0}+h_{0}}^{p}$ and in $L_{\varepsilon_{0},g_{0}+h_{0}}^{p}$
for $2\le p<2^{*}$, so $|u_{n}|^{p-2}u_{n}\rightarrow|u_{0}|^{p-2}u_{0}$
in $L_{\varepsilon_{0},g_{0}+h_{0}}^{p'}$ and, by continuity of $A_{h_{0}}^{\varepsilon_{0}}$,\begin{equation}
i_{\varepsilon_{0},g_{0}+h_{0}}^{*}(|u_{n}|^{p-2}u_{n})=A_{h_{0}}^{\varepsilon_{0}}(|u_{n}|^{p-2}u_{n})
\rightarrow A_{h_{0}}^{\varepsilon_{0}}(|u_{0}|^{p-2}u_{0})\
in H_{\varepsilon_{0},g_{0}+h_{0}}^{1}=H_{\varepsilon_{0},g_{0}}^{1}.\label{eq:stella}\end{equation}
By the reguarity of the map $A$ we have, for some $\theta\in(0,1)$
\begin{equation}
\begin{array}{c}
\|A_{h_{n}}^{\varepsilon_{n}}(|u_{n}|^{p-2}u_{n})-
A_{h_{0}}^{\varepsilon_{0}}(|u_{n}|^{p-2}u_{n})\|_{H_{\varepsilon_{0},g_{0}}^{1}}
\leq\|u_{n}\|_{L_{\varepsilon_{0},g_{0}}^{p'}}^{p-1}\left[|\varepsilon_{n}-\varepsilon_{0}|+
\|h_{n}-h_{0}\|_{k}\right]\times\\
\times\|A'(\varepsilon_{0}+\theta(\varepsilon_{n}-\varepsilon_{0}),h_{0}
+\theta(h_{n}-h_{0}))\|_{\mathcal{L}((0,1)\times\mathscr{B}_{\rho},
\mathcal{L}(L_{\varepsilon_{0},g_{0}}^{p'},H_{\varepsilon_{0},g_{0}}^{1}))}.\end{array}\label{eq:doppiastella}
\end{equation}
By (\ref{eq:stella}) and (\ref{eq:doppiastella}) we get that $A_{h_{n}}^{\varepsilon_{n}}(|u_{n}|^{p-2}u_{n})
\rightarrow A_{h_{0}}^{\varepsilon_{0}}(|u_{0}|^{p-2}u_{0})$
in $H_{\varepsilon_{0},g_{0}}^{1}$ then in $H_{g_{0}}^{\tau}$. Since
\[
0=F(\varepsilon_{n},h_{n},u_{n})=u_{n}-A_{h_{n}}^{\varepsilon_{n}}(|u_{n}|^{p-2}u_{n})\]
 we get the claim.\end{proof}
\begin{lem}
\label{lem:4.3}For any $(\varepsilon_{0},h_{0},u_{0})\in(0,1)\times\mathscr{B}_{\rho}\times H_{g_{0}}^{\tau}\smallsetminus\left\{ 0\right\} $
such that $F(\varepsilon_{0},h_{0},u_{0})=0$, it holds that, if $w\in\ker F'_{u}(\varepsilon_{0},h_{0},u_{0})$
and \[
\langle F'_{\varepsilon,h}(\varepsilon_{0},h_{0},u_{0})\left[\varepsilon,h\right],w\rangle_{H_{\varepsilon_{0},g_{0}+h_{0}}^{1}}=0\ \forall\varepsilon\in\mathbb{R},\ h\in\mathscr{S}^{k},\]
then $w=0$.\end{lem}
\begin{proof}
\emph{Step 1. }By the definition of $F$ and Lemma \ref{lem:derivate}
we get\begin{equation}
F'_{\varepsilon,h}(\varepsilon_{0},h_{0},u_{0})\left[\varepsilon,h\right]=-A'(\varepsilon_{0},h_{0})\left[\varepsilon,h\right]\left(|u_{0}|^{p-2}u_{0}\right)\label{eq:4.3.1}\end{equation}
and so \begin{eqnarray*}
 &  & \langle F'_{\varepsilon,h}(\varepsilon_{0},h_{0},u_{0})\left[\varepsilon,h\right],w\rangle_{H_{g_{0}+h_{0},\varepsilon_{0}}^{1}}\\
 & = & -E_{h_{0}}^{\varepsilon_{0}}\left(A'(\varepsilon_{0},h_{0})\left[\varepsilon,h\right]\left(|u_{0}|^{p-2}u_{0}\right),w\right)=\\
 & = & -G'(\varepsilon_{0},h_{0})\left[\varepsilon,h\right]\left(|u_{0}|^{p-2}u_{0},w\right)+E'(\varepsilon_{0},h_{0})\left[\varepsilon,h\right]\left(u_{0},w\right)=\\
 & = & -\frac{1}{2\varepsilon_{0}^{n}}\int_{M}tr(g^{-1}h)|u_{0}|^{p-2}u_{0}wd\mu_{g}+\frac{n\varepsilon}{\varepsilon_{0}^{n+1}}\int_{M}|u_{0}|^{p-2}u_{0}wd\mu_{g}\\
 &  & +\frac{1}{2\varepsilon_{0}^{n}}\int_{M}tr(g^{-1}h)u_{0}wd\mu_{g}+\frac{1}{\varepsilon_{0}^{n-2}}\int_{M}\langle\nabla_{g}u_{0},\nabla_{g}w\rangle_{b(h)}d\mu_{g}\\
 &  & -\frac{n\varepsilon}{\varepsilon_{0}^{n+1}}\int_{M}u_{0}wd\mu_{g}-\frac{(n-2)\varepsilon}{\varepsilon_{0}^{n-1}}\int_{M}\langle\nabla_{g}u_{0},\nabla_{g}w\rangle d\mu_{g}.\end{eqnarray*}
Here we use that $A_{h_{0}}^{\varepsilon_{0}}(|u_{0}|^{p-2}u_{0})=u_{0}$.
Moreover $g=g_{0}+h_{0}$ with $h_{0}\in\mathscr{B}_{\rho}$ and $b(h):=\frac{1}{2}tr(g^{-1}h)g-g^{-1}hg^{-1}$. 

If we choose $\varepsilon=0$, by the previous equation we get\begin{eqnarray}
\langle F'_{\varepsilon,h}(\varepsilon_{0},h_{0},u_{0})\left[0,h\right],w\rangle_{H_{\varepsilon_{0},g_{0}+h_{0}}^{1}} & = & \frac{1}{2\varepsilon_{0}^{n}}\int_{M}tr(g^{-1}h)\left[u_{0}-|u_{0}|^{p-2}u_{0}\right]wd\mu_{g}+\nonumber\\
 &  & +\frac{1}{\varepsilon_{0}^{n-2}}\int_{M}\langle\nabla_{g}u_{0},\nabla_{g}w\rangle_{b(h)}d\mu_{g} \label{eq:4.3.3}\end{eqnarray}

\emph{Step 2. }We prove that, if $\langle F'_{\varepsilon,h}(\varepsilon_{0},h_{0},u_{0})\left[0,h\right],w
\rangle_{H_{\varepsilon_{0},g_{0}+h_{0}}^{1}}=0$
$\forall h\in\mathscr{S}^{k}$, then $ $it holds\[
\langle\nabla_{g}u_{0}(\xi),\nabla_{g}w(\xi)\rangle_{b(h)}=0\text{ for all }\xi\in M.\]
Given $\xi_{0}\in M$, we consider the normal coordinates at $\xi_{0}$
and we set\[
\tilde{u}_{0}(x)=u_{0}(\exp_{\xi_{0}}x),\ \tilde{w}(x)=w(\exp_{\xi_{0}}x),\text{ for }x\in B(0,R)\subset\mathbb{R}^{n}.\]
We will prove that ${\displaystyle \frac{\partial\tilde{u}_{0}(0)}{\partial x_{1}}\frac{\partial\tilde{w}(0)}{\partial x_{2}}+\frac{\partial\tilde{u}_{0}(0)}{\partial x_{2}}\frac{\partial\tilde{w}(0)}{\partial x_{1}}=0}$.
Analogously we can get ${\displaystyle \frac{\partial\tilde{u}_{0}(0)}{\partial x_{i}}\frac{\partial\tilde{w}(0)}{\partial x_{j}}+\frac{\partial\tilde{u}_{0}(0)}{\partial x_{j}}\frac{\partial\tilde{w}(0)}{\partial x_{i}}=0}$.

If $\xi_{0}\neq\tau\xi_{0}$, we assume that $B_{g}(\xi_{0},R)\cap B_{g}(\tau\xi_{0},R)=\emptyset$.
Then choosing $h\in\mathscr{S}^{k}$ vanishing outside $B_{g}(\xi_{0},R)\cup B_{g}(\tau\xi_{0},R)$,
by the fact that $h(\tau x)=h(x)$ on $M$, by (\ref{eq:4.4.3}) and
by our assumption we have \begin{equation}
\frac{1}{\varepsilon_{0}^{n}}\int_{B(\xi_{0},R)}tr(g^{-1}h)\left[u_{0}-|u_{0}|^{p-2}u_{0}\right]wd\mu_{g}+\frac{1}{\varepsilon_{0}^{n-2}}\int_{M}\langle\nabla_{g}u_{0},\nabla_{g}w\rangle_{b(h)}d\mu_{g}=0.\label{eq:4.3.4a}\end{equation}
Using the normal coordinates at $\xi_{0}$ we choose $h$ such that
the matrix $\{h_{ij}(x)\}_{i,j=1,\cdots,n}$ has the form $h_{12}(x)=h_{21}(x)\in C_{0}^{\infty}(B(0,R))$
and $h_{ij}\equiv0$ otherwise. By (\ref{eq:4.3.3}) we have\begin{eqnarray}
0 & = & \int_{B(0,R)}|g(x)|^{1/2}h_{12}(x)\left\{ -\varepsilon_{0}^{2}b_{12}(x)\left(\frac{\partial\tilde{u}_{0}}{\partial x_{1}}\frac{\partial\tilde{w}}{\partial x_{2}}+\frac{\partial\tilde{u}_{0}}{\partial x_{2}}\frac{\partial\tilde{w}}{\partial x_{1}}\right)+\sigma(x)\right\} dx\label{eq:4.3.4}\end{eqnarray}
where 
\begin{multline}
\sigma(x)=-\varepsilon_{0}^{2}\sum_{\begin{array}{c}
r,s=1,\dots,n\\
(r,s)\neq(1,2)\\
(r,s)\neq(2,1)\end{array}}
b_{rs}\left(\frac{\partial\tilde{u}_{0}}{\partial x_{r}}\frac{\partial\tilde{w}}{\partial x_{s}}\right)\\
+2g^{12}\left\{ \frac{\varepsilon_{0}^{2}}{2}\sum_{i,j=1}^{n}g^{ij}
\left(\frac{\partial\tilde{u}_{0}}{\partial x_{i}}\frac{\partial\tilde{w}}{\partial x_{j}}\right)
+\left[\tilde{u}_{0}-|\tilde{u}_{0}|^{p-2}\tilde{u}_{0}\right]\tilde{w}\right\} .
\end{multline}
Here $b_{rs}(x)=\left(g^{-1}(x)\Gamma g^{-1}(x)\right)_{rs}$, where
$\Gamma_{12}=\Gamma_{21}=0$, $\Gamma_{ij}=\Gamma_{j,i}=0$ for $i,j=1,\dots,n$,
$(i,j)\neq(1,2)$. Then $b_{12}(0)=b_{21}(0)=1$, $b_{rs}(0)=0$ otherwise,
so $\sigma(0)=0$. By (\ref{eq:4.3.4}), at this point we have \[
-\varepsilon_{0}^{2}b_{12}(x)\left(\frac{\partial\tilde{u}_{0}}{\partial x_{1}}(x)\frac{\partial\tilde{w}}{\partial x_{2}}(x)+\frac{\partial\tilde{u}_{0}}{\partial x_{2}}(x)\frac{\partial\tilde{w}}{\partial x_{1}}(x)\right)+\sigma(x)\text{ for }x\in B(0,R).\]
Then \[
\frac{\partial\tilde{u}_{0}}{\partial x_{1}}(0)\frac{\partial\tilde{w}}{\partial x_{2}}(0)+\frac{\partial\tilde{u}_{0}}{\partial x_{2}}(0)\frac{\partial\tilde{w}}{\partial x_{1}}(0)=0.\]
If $\xi_{0}=\tau\xi_{0}$, we consider the equality (\ref{eq:4.3.3})
when $h\in\mathscr{S}^{k}$ vanishes outside $B_{g}(\xi_{0},R)$,
recalling that $h(\tau(\xi))=h(\xi)$ for $\xi\in M$. Arguing as
in the previous case, by (\ref{eq:4.3.4})we get that \[
\gamma(x)=\varepsilon_{0}^{2}b_{12}(x)\left(\frac{\partial\tilde{u}_{0}}{\partial x_{1}}\frac{\partial\tilde{w}}{\partial x_{2}}+\frac{\partial\tilde{u}_{0}}{\partial x_{2}}\frac{\partial\tilde{w}}{\partial x_{1}}\right)+\sigma(x)\]
 is antisymmetric with respect to $\bar{\tau}=\exp_{\xi_{0}}^{-1}\tau\exp_{\xi_{0}}$.
Also, we have that $\gamma$ is symmetric with respect to $\bar{\tau}$,
so $\gamma(0)=0$, and, since $b_{12}(0)=1$ and $\sigma(0)=0$, we
have again ${\displaystyle \frac{\partial\tilde{u}_{0}}{\partial x_{1}}(0)\frac{\partial\tilde{w}}{\partial x_{2}}(0)+\frac{\partial\tilde{u}_{0}}{\partial x_{2}}(0)\frac{\partial\tilde{w}}{\partial x_{1}}(0)=0}$. 

Now we prove that ${\displaystyle \frac{\partial\tilde{u}_{0}(0)}{\partial x_{i}}\frac{\partial\tilde{w}(0)}{\partial x_{i}}=0}$
for all $i=1,\dots,n$. 

If $\xi_{0}\neq\tau\xi_{0}$, arguing as in the previous case we get
(\ref{eq:4.3.4a}). This time we choose the matrix $\{h_{ij}(x)\}_{i,j}$
such that $h_{11}\in C_{0}^{\infty}(B(0,R))$, $h_{22}=-h_{11}$ and
$h_{ij}\equiv0$ otherwise. Because $tr(g^{-1}h)=(g^{11}-g^{22})h_{11}$,
by (\ref{eq:4.3.4a}), we get \begin{eqnarray}
0 & = & \int_{B(0,R)}|g(x)|^{1/2}h_{11}(x)\Big\{\left[g^{11}(x)-g^{22}(x)\right]\times\nonumber\\
&&\times \left(\varepsilon_{0}^{2}\sum_{ij}g^{ij}(x)\frac{\partial\tilde{u}_{0}}{\partial x_{i}}\frac{\partial\tilde{w}}{\partial x_{j}}+\tilde{u}_{0}\tilde{w}-|\tilde{u}_{0}|^{p-2}\tilde{u}_{0}\tilde{w}\right)\nonumber \\
 &  & -\varepsilon_{0}^{2}\left[g^{11}(x)g^{12}(x)-g^{12}(x)g^{21}(x)\right]\left(\frac{\partial\tilde{u}_{0}}{\partial x_{1}}\frac{\partial\tilde{w}}{\partial x_{2}}+\frac{\partial\tilde{u}_{0}}{\partial x_{2}}\frac{\partial\tilde{w}}{\partial x_{1}}\right)\label{eq:4.3.5}\\
 &  & -\varepsilon_{0}^{2}\sum_{k=1}^{n}\left[\left(g^{1k}(x)\right)^{2}-\left(g^{2k}(x)\right)^{2}\right]\frac{\partial\tilde{u}_{0}}{\partial x_{k}}\frac{\partial\tilde{w}}{\partial x_{k}}\Big\}dx.\nonumber \end{eqnarray}
Then, recalling that ${\displaystyle \frac{\partial\tilde{u}_{0}}{\partial x_{1}}(0)\frac{\partial\tilde{w}}{\partial x_{2}}(0)+\frac{\partial\tilde{u}_{0}}{\partial x_{2}}(0)\frac{\partial\tilde{w}}{\partial x_{1}}(0)=0}$
and that $g^{ij}(0)=\delta_{ij}$we have

\[
\left[\left(g^{11}(0)\right)^{2}-\left(g^{21}(0)\right)^{2}\right]\frac{\partial\tilde{u}_{0}(0)}{\partial x_{1}}\frac{\partial\tilde{w}(0)}{\partial x_{1}}+\left[\left(g^{12}(0)\right)^{2}-\left(g^{22}(0)\right)^{2}\right]\frac{\partial\tilde{u}_{0}(0)}{\partial x_{2}}\frac{\partial\tilde{w}(0)}{\partial x_{2}}=0.\]
So ${\displaystyle \frac{\partial\tilde{u}_{0}(0)}{\partial x_{1}}\frac{\partial\tilde{w}(0)}{\partial x_{1}}=\frac{\partial\tilde{u}_{0}(0)}{\partial x_{2}}\frac{\partial\tilde{w}(0)}{\partial x_{2}}}$
and, analogously, ${\displaystyle \frac{\partial\tilde{u}_{0}(0)}{\partial x_{1}}\frac{\partial\tilde{w}(0)}{\partial x_{1}}=\frac{\partial\tilde{u}_{0}(0)}{\partial x_{i}}\frac{\partial\tilde{w}(0)}{\partial x_{i}}}$
for all $i$. At this point, since ${\displaystyle \frac{\partial\tilde{u}_{0}(0)}{\partial x_{i}}\frac{\partial\tilde{w}(0)}{\partial x_{j}}+\frac{\partial\tilde{u}_{0}(0)}{\partial x_{j}}\frac{\partial\tilde{w}(0)}{\partial x_{i}}=0}$
for all $i\neq j$ we get \[
{\displaystyle \frac{\partial\tilde{u}_{0}(0)}{\partial x_{i}}\frac{\partial\tilde{w}(0)}{\partial x_{i}}=0}\text{ for all }i=1,\cdots,n.\]

If $\xi_{0}=\tau\xi_{0}$, since $h$ is symmetric with respect to
$\tau$, by (\ref{eq:4.3.5}) we get that \begin{eqnarray*}
\gamma(x) & = & \left[g^{11}(x)-g^{22}(x)\right]\left(\varepsilon_{0}^{2}\sum_{ij}g^{ij}(x)\frac{\partial\tilde{u}_{0}}{\partial x_{i}}\frac{\partial\tilde{w}}{\partial x_{j}}+\tilde{u}_{0}\tilde{w}-|\tilde{u}_{0}|^{p-2}\tilde{u}_{0}\tilde{w}\right)\\
 &  & -\varepsilon_{0}^{2}\left[g^{11}(x)g^{12}(x)-g^{12}(x)g^{21}(x)\right]\left(\frac{\partial\tilde{u}_{0}}{\partial x_{1}}\frac{\partial\tilde{w}}{\partial x_{2}}+\frac{\partial\tilde{u}_{0}}{\partial x_{2}}\frac{\partial\tilde{w}}{\partial x_{1}}\right)\\
 &  & -\varepsilon_{0}^{2}\sum_{k=1}^{n}\left[\left(g^{1k}(x)\right)^{2}-\left(g^{2k}(x)\right)^{2}\right]\frac{\partial\tilde{u}_{0}}{\partial x_{k}}\frac{\partial\tilde{w}}{\partial x_{k}}\end{eqnarray*}
is antisymmetric with respect to $\bar{\tau}=\exp_{\xi_{0}}^{-1}\tau\exp_{\xi_{0}}$.
Concluding \[
0=\gamma(0)=\left(g^{11}(0)\right)^{2}\frac{\partial\tilde{u}_{0}(0)}{\partial x_{1}}\frac{\partial\tilde{w}(0)}{\partial x_{1}}-\left(g^{22}(0)\right)^{2}\frac{\partial\tilde{u}_{0}(0)}{\partial x_{2}}\frac{\partial\tilde{w}(0)}{\partial x_{2}}.\]
At this point, arguing as above we have that \[
{\displaystyle \frac{\partial\tilde{u}_{0}(0)}{\partial x_{i}}\frac{\partial\tilde{w}(0)}{\partial x_{i}}=0}\text{ for all }i=1,\cdots,n.\]
 and the Step 2 is proved.

\emph{Step 3}. Conclusion of the proof.

By Step 2, we have that, for any $h\in\mathscr{S}^{k}$\begin{equation}
0=\langle F'_{\varepsilon,h}(\varepsilon_{0},h_{0},u_{0})\left[0,h\right],w\rangle_{H_{\varepsilon_{0},g_{0}+h_{0}}^{1}}=\frac{1}{2\varepsilon_{0}^{n}}\int_{M}tr(g^{-1}h)u_{0}\left(1-|u_{0}|^{p-2}\right)wd\mu_{g}.\label{eq:4.3.7}\end{equation}
 Here $g=g_{0}+h_{0}$. Moreover it holds\begin{eqnarray}
-\varepsilon_{0}\Delta_{g}w+w=(p-1)|u_{0}|^{p-2}w &  & w\in H_{g}^{\tau}\label{eq:4.3.9}\end{eqnarray}
We choose $h(\xi)=\alpha(\xi)g(\xi)$ for any $\alpha\in C^{\infty}(M)$
with $\alpha(\tau\xi)=\alpha(\xi)$, so, by (\ref{eq:4.3.7}), the
function $u_{0}\left(1-|u_{0}|^{p-2}\right)w$ is antisymmetric with
respect to the involution $\tau$. Furthermore $u_{0}\left(1-|u_{0}|^{p-2}\right)w$
is also symmetric, so \begin{equation}
u_{0}\left(1-|u_{0}|^{p-2}\right)w\equiv0.\label{eq:4.3.10}\end{equation}
By contradiction we assume that $w$ does not vanish indentically
in $M$. Since $u_{0}\in H_{g}^{\tau}\smallsetminus\left\{ 0\right\} $
we can split \[
M=M^{0}\cup M^{1}\cup\tau M^{1}\cup M^{+}\cup\tau M^{+}\]
where $M^{0}=\left\{ x\in M\ :\ u_{0}(x)=0\right\} $, $M^{1}=\left\{ x\in M\ :\ u_{0}(x)=1\right\} $,
and $M^{+}=\left\{ x\in M\ :\ u_{0}(x)>0,\ u_{0}(x)\neq1\right\} $.
By (\ref{eq:4.3.10}) we have that $w\equiv0$
on the open subset $M^{+}\cup\tau M^{+}$. Also, we notice that $M_0$ and $M_1$ are disjoint sets because 
$u_0$ is a continuous funcion. By this, and by (\ref{eq:4.3.9}), we have that 
$-\varepsilon_{0}\Delta_{g}w+w=0$ on $M_0$ and $w=0$  on $\partial M_0$. By the maximum 
principle, we conclude that $w=0$ on $M_0$.
So we have that, by (\ref{eq:4.3.9}), 
$-\varepsilon_{0}\Delta_{g}w+w=(p-1)w$
on the whole $M$. On the other hand, by \cite{AKS}, we have that $\mu_{g}\left(\left\{ x\in M\ :\ w(x)=0\right\} \right)=0$.
A contradiction arises and that concludes the proof 
\end{proof}
With the same argument we can prove the following corollary.
\begin{cor}
\label{cor:teo10}Given $ $$\varepsilon_{0}$, for any $(h_{0},u_{0})\in\mathscr{B}_{\rho}\times H_{g_{0}}^{\tau}\smallsetminus\left\{ 0\right\} $
such that $F(\varepsilon_{0},h_{0},u_{0})=0$, if $w\in\ker F'_{u}(\varepsilon_{0},h_{0},u_{0})$
and \[
\langle F'_{h}(\varepsilon_{0},h_{0},u_{0})\left[h\right],w\rangle_{H_{\varepsilon_{0},g_{0}+h_{0}}^{1}}=0\ \forall h\in\mathscr{S}^{k},\]
then $w=0$.\end{cor}
\begin{lem}
\label{lem:4.4}Given $g_{0}\in\mathscr{M}^{k}$ and $\varepsilon_{0}$,
if there exists a number $\mu>m_{\varepsilon_{0},g_{0}}$ not a critical
level of the functional $J_{\varepsilon_{0},g_{0}}$, then, for $\rho$
small enough, the map $\pi\circ i:G^{-1}(0)\rightarrow\mathscr{S}^{k}$
is proper. Here $G$ is defined in (\ref{eq:G}), $i$ is the canonical
embedding from $G^{-1}(0)$ into $\mathscr{S}^{k}\times H_{g_{0}}^{\tau}$
and $\pi$ is the projection from $\mathscr{S}^{k}\times H_{g_{0}}^{\tau}$
into $\mathscr{S}^{k}$.\end{lem}
\begin{proof}
Let $\left\{ u_{n}\right\} \subset\mathscr{D}$, where 

\[
\mathscr{D}=\left\{ u\in H_{g_{0}}^{\tau}\ :\ \mu_{0}<J_{\varepsilon_{0},g_{0}}(u)<\mu\right\} ,\]
 and $\mu_{0}$ is an arbitrary number in $(0,m_{\varepsilon_{0},g_{0}}^{\tau})$.
It is sufficient to prove that if $u_{n}$ satifisfies $-\varepsilon_{0}^{2}\Delta_{g_{0}+h_{n}}u_{n}+u_{n}=|u_{n}|^{p-2}u_{n}$
with $h_{n}\rightarrow h_{0}\in\mathscr{B}_{\rho},$ then the sequence
$\left\{ u_{n}\right\} $ has a subsequence convergent in $\mathscr{D}$.
First we show that $\left\{ u_{n}\right\} $ is bounded in $H_{g_{0}}^{\tau}$.
Since the sets $H_{g_{0}+h}^{1}(M)$ and $H_{g_{0}}^{1}(M)$ are the
same in $h\in\mathscr{B}_{\rho}$ and the norms $\|\cdot\|_{H_{g_{0}+h}^{1}}$
and $\|\cdot\|_{H_{g_{0}+h}^{1}}$ are equivalent with equivalence
constants $c_{1}$ and $c_{2}$ not depending on $h$, we have \[
c_{1}\|u\|_{H_{\varepsilon_{0},g_{0}}^{1}}\leq\|u\|_{H_{\varepsilon_{0},g_{0}+h}^{1}}\leq c_{2}\|u\|_{H_{\varepsilon_{0},g_{0}}^{1}}.\]
By this, and because $u_{n}\in\mathcal{N}_{\varepsilon_{0},g_{0}+h_{n}}^{\tau}$
we have \begin{eqnarray}
\left(\frac{1}{2}-\frac{1}{p}\right)c_{1}^{2}\|u_{n}\|_{H_{\varepsilon_{0},g_{0}}^{1}}^{2} & \le & \left(\frac{1}{2}-\frac{1}{p}\right)\|u_{n}\|_{H_{\varepsilon_{0},g_{0}+h_{n}}^{1}}^{2}=J_{\varepsilon_{0},g_{0}+h_{n}}(u_{n})=\nonumber \\
 & = & \frac{1}{2}\|u_{n}\|_{H_{\varepsilon_{0},g_{0}+h_{n}}^{1}}^{2}-\frac{1}{p}\|u_{n}\|_{L_{\varepsilon_{0},g_{0}+h_{n}}^{p}}^{p}\nonumber \\
 & \leq & J_{\varepsilon_{0},g_{0}}(u_{n})+c\|h_{n}\|_{k}\left[\|u_{n}\|_{H_{\varepsilon_{0},g_{0}}^{1}}^{2}+\|u_{n}\|_{L_{\varepsilon_{0},g_{0}}^{p}}^{p}\right]\le\nonumber \\
 & \leq & \mu+c\rho\left[\|u_{n}\|_{H_{\varepsilon_{0},g_{0}}^{1}}^{2}+\|u_{n}\|_{L_{\varepsilon_{0},g_{0}}^{p}}^{p}\right]\label{eq:4.4.1}\end{eqnarray}
Moreover, since $\mu_{0}<J_{\varepsilon_{0},g_{0}}(u_{n})<\mu$ we
get\begin{equation}
\mu_{0}<\left(\frac{1}{2}-\frac{1}{p}\right)\|u_{n}\|_{L_{\varepsilon_{0},g_{0}}^{p}}^{p}<\mu\label{eq:4.4.2}\end{equation}
by (\ref{eq:4.4.1}) and (\ref{eq:4.4.2}), if $\|u_{n}\|_{H_{\varepsilon_{0},g_{0}}^{1}}\rightarrow+\infty$
we get \[
\left(\frac{1}{2}-\frac{1}{p}\right)\|u_{n}\|_{H_{\varepsilon_{0},g_{0}}^{1}}^{2}\leq\mu+c\rho\|u_{n}\|_{H_{\varepsilon_{0},g_{0}}^{1}}^{2}+c\rho,\]
 then, choosing $\rho$ small enough, we get the contradiction.

Since the sequence $\left\{ u_{n}\right\} $ is bounded in $H_{g_{0}}^{\tau}$
and $H_{g_{0}+h_{0}}^{\tau}$, up to a subsequence $u_{n}\rightarrow u$
in $L_{g_{0}+h_{0}}^{t,\tau}(M)$ and $L_{g_{0}}^{t,\tau}(M)$ for
$2\le t<2^{*}$. Then \begin{equation}
i_{\varepsilon_{0},g_{0}+h_{0}}^{*}\left(|u_{n}|^{p-2}u_{n}\right)=
A_{h_{0}}^{\varepsilon_{0}}\left(|u_{n}|^{p-2}u_{n}\right)
\rightarrow A_{h_{0}}^{\varepsilon_{0}}\left(|u|^{p-2}u\right)\text{ in }H_{\varepsilon_{0},g_{0}+h_{0}}^{\tau}\label{eq:4.4.3}\end{equation}
Since the map $A$ is of class $C^{1}$ (see Lemma \ref{lem:derivate})
we have, for some $\theta\in(0,1)$\begin{eqnarray}
 &  & \|A_{h_{n}}^{\varepsilon_{0}}\left(|u_{n}|^{p-2}u_{n}\right)
 -A_{h_{0}}^{\varepsilon_{0}}\left(|u_{n}|^{p-2}u_{n}\right)\|_{H_{\varepsilon_{0},g_{0}}^{1}}\nonumber \\
 & = & \|A'(\varepsilon_{0},h_{0}+\theta(h_{n}-h_{0}))
 \left[0,h_{n}-h_{0}\right]\left(|u_{n}|^{p-2}u_{n}\right)\|_{H_{\varepsilon_{0},g_{0}}^{1}}\label{eq:4.4.4}\\
 & \le & \||u_{n}|^{p-1}\|_{L_{\varepsilon_{0},g_{0}}^{p'}}\|h_{n}-h_{0}\|_{k}\|A'(\varepsilon_{0},h_{0}
 +\theta(h_{n}-h_{0}))\|_{\mathcal{L}(\mathscr{B}_{\rho},\mathcal{L}(L_{\varepsilon_{0},g_{0}}^{p',\tau},
 H_{\varepsilon_0,g_{0}}^{\tau}))}\nonumber \end{eqnarray}
By (\ref{eq:4.4.3}) and (\ref{eq:4.4.4}) we get $A_{h_{n}}^{\varepsilon_{0}}\left(|u_{n}|^{p-2}u_{n}\right)
\rightarrow A_{h_{0}}^{\varepsilon_{0}}\left(|u|^{p-2}u\right)$
in $H_{\varepsilon_{0},g_{0}}^{\tau}$. 

Since $0=u_{n}-A_{h_{n}}^{\varepsilon_{0}}\left(|u_{n}|^{p-2}u_{n}\right)$
we get that $u_{n}$ converges to $u$ in $H_{g_{0}}^{\tau}$. Moreover
$u-A_{h_{0}}^{\varepsilon_{0}}\left(|u|^{p-2}u\right)=0$. Since $\mu_{0}$
and $\mu$ are not critical values for $J_{\varepsilon_{0},g_{0}}(u)$,
we have that $\mu_{0}<J_{\varepsilon_{0},g_{0}}(u)<\mu$. Then $u\in\mathscr{D}$. 
\end{proof}

\section{\label{sec:application}An application}

In this section we choose  $\tau=-\Id$ and the manifold $M$
invariant with respect to the involution $\tau=-\Id$. We also assume
$0\not\in M$, so $M_{\tau}=\emptyset$. 
Using the previous results of genericity for non degenerate sign changing solutions of problem (\ref{eq:Ptau}) 
we obtain a lower bound on the number of non degenerate solutions which change sign exactly once.
This estimate is formulated also in \cite{MP10}. In the cited paper this result is proved under an assumption 
on non degeneracy of critical points that we do not need.

We sketch the proof of propositions \ref{thm:T1} and \ref{thm:T3} showing how we use the results of 
genericity for non degeneracy of critical points to obtain the same estimate.

We recall that there exists a unique positive spherically symmetric
function $U\in H^{1}(\mathbb{R}^{n})$ such that \textbackslash{}begin
$-\Delta U+U=U^{p-1}$ in $\mathbb{R}^{n}$. Also, it is well known
that for any $\varepsilon>0$, $U_{\varepsilon}(x):=U\left(\frac{x}{\varepsilon}\right)$
is a solution of $-\varepsilon^{2}\Delta U_{\varepsilon}+U_{\varepsilon}=U_{\varepsilon}^{p-1}$
in $\mathbb{R}^{n}$.

Let $g_{0}$ be in $\mathscr{M}^{k}$ and $h$ be in $\mathscr{B}_{\rho}$
for some $\rho>0$. Let us define a smooth cut off real function $\chi_{R}$
such that $\chi_{R}(t)=1$ if $0\le t\le R/2$, $\chi_{R}(t)=0$ if
$t\ge R$ and $\left|\chi'(t)\right|<2/R$. Fixed $q\in M$ and $\varepsilon>0$
we define on $M$ the function\[
W_{q,\varepsilon}^{g}(x)=\left\{ \begin{array}{ccc}
U_{\varepsilon}(\exp_{q}^{-1}(x))\chi_{R}(|\exp_{q}^{-1}(x)|) &  & \text{if }x\in B_{g}(q,R)\\
0 &  & \text{otherwise}\end{array}\right.,\]
where $B_{g}(q,R)$ is the geodesic ball of radius $R$ centered at
$q$. We choose $R$ smaller than the injectivity radius of $M$ and
such that $B_{g}(q,R)\cap B_{g}(-q,R)=\emptyset$. Here and in the
following we set $g=g_{0}+h$. 

We can define a map $\Phi_{\varepsilon,g}:M\rightarrow\mathcal{N}_{\varepsilon,g}^{\tau}$
as \[
\Phi_{\varepsilon,g}(q)=t\left(W_{q,\varepsilon}^{g}\right)W_{q,\varepsilon}^{g}-t\left(W_{-q,\varepsilon}^{g}\right)W_{-q,\varepsilon}^{g}.\]
Here \[
\left[t\left(W_{q,\varepsilon}^{g}\right)\right]^{p-2}=\frac{\int_{M}\varepsilon^{2}|\nabla_{g}W_{q,\varepsilon}^{g}|^{2}+|W_{q,\varepsilon}^{g}|^{2}d\mu_{g}}{\int|W_{q,\varepsilon}^{g}|^{p}d\mu_{g}},\]
thus $t\left(W_{q,\varepsilon}^{g}\right)W_{q,\varepsilon}^{g}\in\mathcal{N}_{\varepsilon,g}$
and we have $\Phi_{\varepsilon,g}(q)=-\Phi_{\varepsilon,g}(-q)$.
Now we can define \begin{eqnarray*}
\tilde{\Phi}_{\varepsilon,g} & : & M/G\rightarrow\mathcal{N}_{\varepsilon,g}^{\tau}/\mathbb{Z}_{2}\\
\tilde{\Phi}_{\varepsilon,g}\left[q\right] & = & \left[\Phi_{\varepsilon,g}(q)\right]=
\left\{ \Phi_{\varepsilon,g}(q),\Phi_{\varepsilon,g}(-q)\right\} \end{eqnarray*}
where\begin{eqnarray*}
M/G=\left\{ \left[q\right]=(q,-q)\ :\ q\in M\right\}  &  & \mathcal{N}_{\varepsilon,g}^{\tau}/\mathbb{Z}_{2}
=\left\{ (u,-u)\ :\ u\in\mathcal{N}_{\varepsilon,g}^{\tau}\right\} .\end{eqnarray*}
We set $\tilde{J}_{\varepsilon,g}\left[u\right]=J_{\varepsilon,g}(u)$.
Obviously, $\tilde{J}_{\varepsilon,g}:\mathcal{N}_{\varepsilon,g}^{\tau}/\mathbb{Z}_{2}\rightarrow\mathbb{R}$.
\begin{lem}
\label{pro:4.5}For any $\delta>0$ there exists  $\varepsilon_{2}=\varepsilon_{2}(\delta)$
such that, if $\varepsilon<\varepsilon_{2}$ then \[
\tilde{\Phi}_{\varepsilon,g_{0}+h}(\left[q\right])\in\mathcal{N}_{\varepsilon,g_{0}+h}^{\tau}
\cap\tilde{J}_{\varepsilon,g_{0}+h}^{2(m_{\infty}+\delta)}\ \forall h\in\mathscr{B}_{\rho}.\]
Moreover we have that
\[
\lim_{\varepsilon\rightarrow0}m_{\varepsilon,g_{0}+h}^{\tau}=2m_{\infty}\text{ uniformly on }h\in\mathscr{B}_{\rho}.\]

\end{lem}
For a proof of this result we refer to \cite{BBM07}.

For any function $u\in\mathcal{N}_{\varepsilon,g_{0}+h}^{\tau}$ we
define \[
\beta_{g}(u)=\frac{{\displaystyle \int_{M}x(u^{+})^{p}d\mu_{g}}}{{\displaystyle \int_{M}(u^{+})^{p}d\mu_{g}}}\]
 where $g=g_{0}+h$. Also, we define\begin{eqnarray*}
\tilde{\beta}_{g} & : & \left(\mathcal{N}_{\varepsilon,g_{0}+h}^{\tau}/\mathbb{Z}_{2}\right)
\cap\tilde{J}_{\varepsilon,g_{0}+h}^{2(m_{\infty}+\delta)}\rightarrow M_{d}/G\\
\tilde{\beta}_{g}(\left[u\right]) & := & \left[\tilde{\beta}_{g}(u)\right]=\left\{ \beta_{g}(u),\beta_{g}(-u)\right\} =
\left\{ \beta_{g}(u),-\beta_{g}(u)\right\} \end{eqnarray*}
where $M_{d}=\left\{ u\in M\ :\ d(x,M)<d\right\} $. 
\begin{lem}
\label{pro:4.6}There exists $\tilde{\delta}$ such that $\forall\delta<\tilde{\delta}$
there exists $\tilde{\varepsilon}=\tilde{\varepsilon}(\delta)$ and
for any $\varepsilon<\tilde{\varepsilon}$ the map \[
\tilde{\beta}_{g}\circ\tilde{\Phi}_{\varepsilon,g}:M/G\stackrel{\tilde{\Phi}_{\varepsilon,g}
}{\longrightarrow}\left(\mathcal{N}_{\varepsilon,g_{0}+h}^{\tau}/\mathbb{Z}_{2}\right)
\cap\tilde{J}_{\varepsilon,g_{0}+h}^{2(m_{\infty}+\delta)}\stackrel{\tilde{\beta}_{g}}{\longrightarrow}M_{d}/G\]
 is continuous and homotopic to identity, for all $g=g_{0}+h$ with
$h\in\mathscr{B}_{\rho}$. 
\end{lem}
For a proof of this result we refer to \cite{BBM07}.

Let us sketch the proof of Proposition \ref{thm:T1}. We are going to
find an estimate on the number of pairs non degenerate critical points
$(u,-u)$ for the functional $J_{\varepsilon,g}:H_{g}^{\tau}\rightarrow\mathbb{R}$
with energy close to $2m_{\infty}$ with respect to the parameters
$(\varepsilon,h)\in(0,\tilde{\varepsilon})\times\mathscr{B}_{\rho}$
for $\tilde{\varepsilon},\rho$ small enough.

We recall that, by Theorem \ref{thm:residuale}, given the positive
numbers $\tilde{\varepsilon},\rho$, the set

\[
D(\tilde{\varepsilon},\rho)=\left\{ \begin{array}{c}
(\varepsilon,h)\in(0,\tilde{\varepsilon})\times\mathscr{B}_{\rho}\ :\text{ any }u\in H_{g_{0}}^{\tau}\text{ solution of}\\
-\varepsilon^{2}\Delta_{g_{0}+h}u+u=|u|^{p-2}u\text{ is not degenerate}\end{array}\right\} \]
is a residual subset in $(0,\tilde{\varepsilon})\times\mathscr{B}_{\rho}$.
Since \[
\lim_{(\varepsilon,h)\rightarrow0}m_{\varepsilon,g_{0}+h}^{\tau}=2m_{\infty},\]
 given $\delta\in\left(0,\frac{m_{\infty}}{4}\right)$, for $(\varepsilon,h)$
small enough we have \[
0<2(m_{\infty}-\delta)<m_{\varepsilon,g_{0}+h}^{\tau}<2(m_{\infty}+\delta)<3m_{\infty},\]
 thus $2(m_{\infty}-\delta)$ is not a critical value of $J_{\varepsilon,g}$
on $H_{g}^{\tau}$. At this point we take $(\varepsilon,h)\in D(\tilde{\varepsilon},\rho)$
with $\tilde{\varepsilon},\rho$ small enough. Thus the critical points
of $J_{\varepsilon,g}$ such that $J_{\varepsilon,g}<3m_{\infty}$
are in a finite number by Theorem \ref{thm:residuale}, and then we
can assume that $2(m_{\infty}+\delta)$ is not a critical value of
$J_{\varepsilon,g}$. 

Let $\mathcal{N}_{\varepsilon}^{\tau}/\mathbb{Z}_{2}$ be the set
obtained by identifying antipodal points of the Nehari manifold $\mathcal{N}_{\varepsilon}^{\tau}$.
It is easy to check that the set $\mathcal{N}_{\varepsilon}^{\tau}/\mathbb{Z}_{2}$
is homeomorphic to the projective space $\mathbb{P}^{\infty}=\Sigma/\mathbb{Z}_{2}$
obtained by identifying antipodal points in $\Sigma$, $\Sigma$ being
the unit sphere in $H_{g}^{\tau}$. We are looking for pairs of nontrivial
critical points $(u,-u)$ of the functional $J_{\varepsilon}:H_{g}^{\tau}\rightarrow\mathbb{R}$,
that is searching for critical points of the functional \begin{eqnarray*}
\tilde{J}_{\varepsilon,g} & : & \left(H_{g}^{\tau}\smallsetminus\left\{ 0\right\} \right)/\mathbb{Z}_{2}\rightarrow\mathbb{R};\\
\tilde{J}_{\varepsilon,g}\left(\left[u\right]\right) & := & J_{\varepsilon,g}(u)=J_{\varepsilon,g}(-u).\end{eqnarray*}
We use the same Morse theory argument as in \cite{BC94}. The following
result can be found in (\cite{Ben95} and Lemma 5.2 of \cite{BC94})
\begin{equation}
P_{t}\left(\tilde{J}_{\varepsilon,g}^{2(m_{\infty}+\delta)},\ \tilde{J}_{\varepsilon,g}^{2(m_{\infty}-\delta)}\right)=
tP_{t}\left(\tilde{J}_{\varepsilon,g}^{2(m_{\infty}+\delta)}\cap\left(\mathcal{N}_{\varepsilon}^{\tau}/
\mathbb{Z}_{2}\right)\right).\label{eq:3.1}\end{equation}
By Lemma \ref{pro:4.5} and Lemma \ref{pro:4.6} we have
that $\tilde{\beta}_{g}\circ\tilde{\Phi}_{\varepsilon,g}:M/G\rightarrow M_{d}/G$
is a map homotopic to the identity of $M/G$ and that $M_{d}/G$ is
homotopic to $M/G$. Therefore we have \begin{equation}
P_{t}\left(\tilde{J}_{\varepsilon,g}^{2(m_{\infty}+\delta)}\cap
\left(\mathcal{N}_{\varepsilon}^{\tau}/\mathbb{Z}_{2}\right)\right)=P_{t}(M/G)+Z(t)\label{eq:3.3}\end{equation}
 were $Z(t)$ is a polynomial with non negative coefficients. Since
the functional $J_{\varepsilon,g}$ satisfies the Palais Smale condition
by the compactness of $M$, and the critical points of $J_{\varepsilon,g}$
in $J_{\varepsilon,g}^{3m_{\infty}}$ are non degenerate (because
$(\varepsilon,h)\in D(\tilde{\varepsilon},\rho)$), by Morse Theory
and relations (\ref{eq:3.1}) and (\ref{eq:3.3}) we get at least
$P_{1}(M/G)$ pairs $(u,-u)$ of non trivial solutions of $-\varepsilon^{2}\Delta_{g}u+u=|u|^{p-2}u$
with $J_{\varepsilon,g}(u)=J_{\varepsilon,g}(-u)<3m_{\infty}$. So,
these solutions change sign exactly once. That concludes the proof
of Proposition \ref{thm:T1}. 
\begin{rem}
\label{thm:T2}In the same way we obtain that, given $g_{0}\in\mathscr{M}^{k}$ and $\varepsilon_{0}>0$,
the set \[
\mathscr{A}^{*}=\left\{ \begin{array}{c}
h\in\mathscr{B}_{\rho}:\text{ the equation }-\varepsilon_{0}^{2}\Delta_{g_{0}+h}u+u=|u|^{p-2}u\\
\text{has at least }P_{1}(M/G)\text{ pairs of non degenerate solutions}\\
(u,-u)\in H_{g}^{\tau}\smallsetminus\left\{ 0\right\} \text{ which change sign exactly once}\end{array}\right\} \]
 is a residual subset of $\mathscr{B}_{\rho}$.
\end{rem}

The proof of Proposition \ref{thm:T3}
can be obtained with similar arguments.


\begin{thebibliography}{10}

\bibitem{AKS}
N.~Aronszajn, A.~Krzywicki, and J.~Szarski, \emph{A unique continuation theorem
  for exterior differential forms on {R}iemannian manifolds}, Ark. Mat.
  \textbf{4} (1962), 417--453 (1962).

\bibitem{Ben95}
V.~Benci, \emph{Introduction to {M}orse theory. {A} new approach}, Topological
  Nonlinear Analysis: Degree, Singularity, and Variations (Michele Matzeu and
  Alfonso Vignoli, eds.), Progress in Nonlinear Differential Equations and
  their Applications, no.~15, Birkh{\"a}user, Boston, 1995, pp.~37--177.

\bibitem{BBM07}
V.~Benci, C.~Bonanno, and A.~M. Micheletti, \emph{On the multiplicity of
  solutions of a nonlinear elliptic problem on {R}iemannian manifolds}, J.
  Funct. Anal. \textbf{252} (2007), no.~2, 464--489.

\bibitem{BC94}
V.~Benci and G.~Cerami, \emph{Multiple positive solutions of some elliptic
  problems via the {M}orse theory and the domain topology}, Calc. Var. Partial
  Differential Equations \textbf{2} (1994), no.~1, 29--48.

\bibitem{BP05}
J.~Byeon and J.~Park, \emph{Singularly perturbed nonlinear elliptic problems on
  manifolds}, Calc. Var. Partial Differential Equations \textbf{24} (2005),
  no.~4, 459--477.

\bibitem{DY}
E.~Dancer and S.~Yan, \emph{Multipeak solutions for a singularly perturbed
  neumann problem}, Pacific J. Math \textbf{189} (1999), no.~2, 241--262.

\bibitem{DMP09}
E.~N. Dancer, A.~M. Micheletti, and A.~Pistoia, \emph{Multipeak solutions for
  some singularly perturbed nonlinear elliptic problems in a {Riemannian}
  manifold}, Manuscripta Math. \textbf{128} (2009), no.~2, 163--193.

\bibitem{DFW}
M.~Del~Pino, P.~Felmer, and J.~Wei, \emph{On the role of mean curvature in some
  singularly perturbed neumann problems}, SIAM J. Math. Anal. \textbf{31}
  (1999), no.~1, 63--79.

\bibitem{GM10}
M.~Ghimenti and A.~M. Micheletti, \emph{On the number of nodal solutions for a
  nonlinear elliptic problem on symmetric {R}iemannian manifolds}, Proceedings
  of the 2007 {C}onference on {V}ariational and {T}opological {M}ethods:
  {T}heory, {A}pplications, {N}umerical {S}imulations, and {O}pen {P}roblems
  (San Marcos, TX), Electron. J. Differ. Equ. Conf., vol.~18, Southwest Texas
  State Univ., 2010, pp.~15--22.

\bibitem{G}
C.~Gui, \emph{Multipeak solutions for a semilinear neumann problem}, Duke Math
  J. \textbf{84} (1996), no.~3, 739--769.

\bibitem{GWW}
C.~Gui, J.~Wei, and M.~Winter, \emph{Multiple boundary peak solutions for some
  singularly perturbed neumann problems}, Ann. Inst. H. Poincar\'{e} Anal. Non
  Lin\'{e}aire \textbf{17} (2000), no.~1, 47--82.

\bibitem{H09}
N.~Hirano, \emph{Multiple existence of solutions for a nonlinear elliptic
  problem in a {Riemannian} manifold}, Nonlinear Anal. \textbf{70} (2009),
  no.~2, 671--692.

\bibitem{Li}
Y.Y. Li, \emph{On a singularly perturbed equation with neumann boundary
  condition}, Comm. Partial Differential Equations \textbf{23} (1998), no.~3-4,
  487--545.

\bibitem{MP09c}
A.~M. Micheletti and A.~Pistoia, \emph{Generic properties of singularly
  perturbed nonlinear elliptic problems on {Riemannian} manifolds}, Adv.
  Nonlinear Stud. \textbf{9} (2009), no.~4, 803--815.

\bibitem{MP09b}
A.~M. Micheletti and A.~Pistoia, \emph{Nodal solutions for a singularly perturbed nonlinear elliptic
  problem in a {Riemannian} manifold}, Adv. Nonlinear Stud. \textbf{9} (2009),
  no.~3, 565--577.

\bibitem{MP09}
A.~M. Micheletti and A.~Pistoia, \emph{The role of the scalar curvature in a nonlinear elliptic problem
  in a {Riemannian} manifold}, Calc. Var. Partial Differential Equations
  \textbf{34} (2009), no.~2, 233--265.

\bibitem{MP10}
A.~M. Micheletti and A.~Pistoia, \emph{On the existence of nodal solutions for a 
nonlinear elliptic problem on symmetric {R}iemannian manifolds}, Int. J. Differ. Equ. (2010), 
Art. ID 432759, 11 pp.

\bibitem{NT1}
W.~N. Ni and I.~Takagi, \emph{On the shape of least-energy solutions to a
  semilinear neumann problem}, Comm. Pure Appl. Math. \textbf{44} (1991),
  no.~7, 819--851.

\bibitem{NT2}
W.~N. Ni and I.~Takagi, \emph{Locating the peaks of least-energy solutions to a semilinear
  neumann problem}, Duke Math. J. \textbf{70} (1993), no.~2, 247--281.

\bibitem{Qu70}
F.~Quinn, \emph{Transversal approximation on {B}anach manifolds}, Global
  {A}nalysis ({P}roc. {S}ympos. {P}ure {M}ath., {V}ol. {XV}, {B}erkeley,
  {C}alif., 1968), Amer. Math. Soc., Providence, R.I., 1970, pp.~213--222.

\bibitem{ST79}
J.-C. Saut and R.~Temam, \emph{Generic properties of nonlinear boundary value
  problems}, Comm. Partial Differential Equations \textbf{4} (1979), no.~3,
  293--319.

\bibitem{Uh76}
K.~Uhlenbeck, \emph{Generic properties of eigenfunctions}, Amer. J. Math.
  \textbf{98} (1976), no.~4, 1059--1078.

\bibitem{V08}
D.~Visetti, \emph{Multiplicity of solutions of a zero-mass nonlinear equation
  in a {Riemannian} manifold}, J. Differential Equations \textbf{245} (2008),
  no.~9, 2397--2439.

\bibitem{W1}
J.~Wei, \emph{On the boundary spike layer solutions to a singularly perturbed
  neumann problem}, J. Differential Equations \textbf{134} (1997), no.~1,
  104--133.

\bibitem{WW05}
J.~Wei and T.~Weth, \emph{On the number of nodal solutions to a singularly
  perturbed {N}eumann problem}, Manuscripta Math. \textbf{117} (2005), no.~3,
  333--344.

\bibitem{WW}
J.~Wei and M.~Winter, \emph{Multipeak solutions for a wide class of singular
  perturbation problems}, J. London Math. Soc. \textbf{59} (1999), no.~2,
  585--606.

\end{thebibliography}
\end{document}